\documentclass[english]{amsart}

\usepackage{amssymb,amsmath,amsthm,esint,braket,mathtools,enumitem,hyperref,comment}

\usepackage{thmtools}
\usepackage{thm-restate}

\usepackage[abbrev,msc-links]{amsrefs}
\usepackage{mathscinet}
\usepackage{doi}

\usepackage[capitalise]{cleveref}

\usepackage[T1]{fontenc}

\usepackage{microtype}

\usepackage[margin=30mm]{geometry}

\providecommand{\abs}[1]{\left\vert #1 \right\vert}
\providecommand{\norm}[1]{\left\Vert #1 \right\Vert}

\providecommand{\pt}[1]{\left( #1 \right)}
\providecommand{\spt}[1]{\left[ #1 \right]}

\newcommand{\RR}{\mathbb R}
\newcommand{\NN}{\mathbb N}

\providecommand{\A}{\mathcal A}
\providecommand{\B}{\mathcal B}
\newcommand{\D}{\,\mathrm d}

\newcommand{\dx}{{\D x}}
\newcommand{\di}{\mathrm{div}}
\newcommand{\ve}{\varepsilon}
\newcommand{\supp}{\mathrm{supp}\,}

\newcommand{\sublim}{\operatornamewithlimits{\longrightarrow}}

\newtheorem{theorem}{Theorem}[section]

\newtheorem{lemma}{Lemma}[section]

\newtheorem{remark}{Remark}[section]

\Crefname{corollary}{Corollary}{Corollaries}
\Crefname{lemma}{Lemma}{Lemmas}
\Crefname{theorem}{Theorem}{Theorems}
\Crefname{proposition}{Proposition}{Propositions}

\DefineSimpleKey{bib}{primaryclass}{}
\DefineSimpleKey{bib}{archiveprefix}{}
\newcommand\myurl[1]{\url{#1}}

\BibSpec{arxiv}{%
	+{}{\PrintAuthors}{author}
	+{,}{ \textit}{title}
	+{}{ \parenthesize}{date}
	+{,}{ arXiv }{eprint}
	+{,}{ \myurl }{url}
}

\begin{document}
	
	\title{Interior regularity of doubly weighted quasi-linear equations}
	\author{Hern{\'a}n Castro}
	\email{hcastro@utalca.cl}
	\address{Instituto de Matem{\'a}ticas, Universidad de Talca, Casilla 747, Talca, Chile}
	\date{\today}
	
	\begin{abstract}
		In this article we study the quasi-linear equation
		\[
		\left\{
		\begin{aligned}
			\di \A(x,u,\nabla u)&=\B(x,u,\nabla u)&&\text{in }\Omega,\\
			u\in H^{1,p}_{loc}&(\Omega;w\dx)
		\end{aligned}
		\right.
		\]
		where \(\A\) and \(\B\) are functions satisfying \(\A(x,u,\nabla u)\sim w_1(\abs{\nabla u}^{p-2}\nabla u+\abs{u}^{p-2}u)\) and \(\B(x,u,\nabla u)\sim w_2(\abs{\nabla u}^{p-2}\nabla u+\abs{u}^{p-2}u)\) for \(p>1\), a \(p\)-admissible weight function \(w_1\), and another weight function \(w_2\) compatible with \(w_1\) in a suitable sense. We establish interior regularity results of weak solutions and use those results to obtain point-wise asymptotic estimates at infinity for solutions to
		\[
		\left\{
		\begin{aligned}
			-\di(w_1\abs{\nabla u}^{p-2}\nabla u)&=w_2\abs{u}^{q-2}u&&\text{in }\Omega,\\
			u\in D^{1,p,w_1}&(\Omega)
		\end{aligned}
		\right.
		\]
		for a critical exponent \(q>p>1\) in the sense of Sobolev.
	\end{abstract}
	
	\maketitle

	\section{Introduction}
	
	This article is a direct continuation of \cite{Cas2023} where we studied qualitative and quantitative properties of weak solutions to the following equation
	\begin{equation}\label{bb-eq}
		\left\{
		\begin{aligned}
			-\di \pt{w_1\abs{\nabla u}^{p-2}\nabla u}&=w_2\abs{u}^{q-2}u&&\text{in }\Omega\\
			u&\in D^{1,p,w_1}(\Omega),
		\end{aligned}
		\right.
	\end{equation}
	for equal weights \(w_1=w_2\) and \(q>p>1\) critical for the weighted Sobolev embedding from \(D^{1,p,w_1}(\Omega)\) into \(L^{q,w_2}(\Omega)\). In this continuation we generalize the results obtained in \cite{Cas2023} for the case of different weights \(w_1\neq w_2\) but satisfying suitable compatibility conditions.
	
	The main motivation behind studying this problem comes from the results in \cite{Cas2021} where the existence to extremals to a Sobolev inequality with monomial weights was analyzed (see also \cites{CR-O2013-2,Cas2016-2}). It is known that extremals to a weighted Sobolev inequality can be viewed as positive solutions to \eqref{bb-eq} for appropriate weights \(w_1,w_2\), and our goal is to obtain as much information as possible regarding said extremals and, in general, of solutions to \eqref{bb-eq}.
	
	As in \cite{Cas2023} the functions \(w_1,w_2\) will be weight functions, meaning locally Lebesgue integrable non-negative function over \(\Omega\subseteq\RR^N\) satisfying at least the following two conditions: if we abuse the notation and we also write \(w\) as the measure induced by \(w\), that is \(w(B)=\int_B w\dx\), we require that \(w\) is a doubling measure in \(\Omega\), meaning that there exists a \emph{doubling constant} \(\gamma>0\) such that
	\begin{equation}\label{doubling-w}
		w(2B)\leq \gamma w(B)
	\end{equation}
	holds for every (open) ball such that \(2B\subset \Omega\), where \(\rho B\) denotes the ball with the same center as \(B\) but with its radius multiplied by \(\rho>0\). The smallest possible \(\gamma>0\) for which \eqref{doubling-w} holds for every ball will be denoted by \(\gamma_w>0\) from now on. Additionally we will suppose that
	\begin{equation}\label{inv-loc-int2}
		0<w<\infty\qquad\lambda-\text{almost everywhere}
	\end{equation}
	where \(\lambda\) denotes the \(N\)-dimensional Lebesgue measure. Observe that these two conditions ensure that the measure \(w\) and the Lebesgue measure \(\lambda\) are absolutely continuous with respect to each other.
	
	In addition to \eqref{doubling-w} and \eqref{inv-loc-int2} we will suppose that the weight \(w_1\) satisfies the following local \((1,p)\) Poincaré inequality: if we write \(\fint_B f w\dx =\frac{1}{w(B)}\int fw\dx\) then
	\begin{enumerate}[label={(P{\scriptsize\Roman*})}]
		\item  \emph{Local weighted \((1,p)\)-Poincaré inequality}: There exists \(\rho\geq 1\) such that if \(u\in C^1(\Omega)\) then for all balls \(B\subset\Omega\) of radius \(l(B)\) one has
		\begin{equation}\label{weighted-poincare}
			\fint_B\abs{u-u_{B,w_1}}w_1\dx\leq C_1l(B)\pt{\fint_{\rho B}\abs{\nabla u}^pw_1\dx}^{\frac1p}
		\end{equation}
		where  
		\begin{equation*}
			u_{B,w}=\fint_B uw\D x
		\end{equation*}
		is the weighted average of \(u\) over \(B\).
	\end{enumerate}
	
	As it can be seen in \cite[Chapter 20]{HeKiMa2006}, when a weight function \(w\) satisfies \eqref{doubling-w}, \eqref{inv-loc-int2} and \eqref{weighted-poincare} then \(w\) is \emph{\(p\)-admissible}, that is, it also satisfies the following properties
	\begin{enumerate}[resume*]
		\item
		\emph{Uniqueness of the gradient}: If \((u_n)_{n\in\NN}\subseteq C^1(\Omega)\) satisfy 
		\[
		\int_\Omega\abs{u_n}^pw_1\dx\sublim\limits_{n\to\infty} 0\quad\text{and}\quad\int_{\Omega}\abs{\nabla u_n-v}^pw_1\dx\sublim\limits_{n\to\infty} 0
		\]
		for some \(v:\Omega\to \RR^N\), then \(v=0\).
		
		\item 
		\emph{Local Poincaré-Sobolev inequality}: There exist constants \(C_3>0\) and \(\chi_1>1\) such that for all balls \(B\subset\Omega\) one has
		\begin{equation}\label{poin-ineq_w1}
			\pt{\fint_B\abs{u-u_{B,w_1}}^{\chi_1 p}w_1\dx}^{\frac1{\chi_1 p}}\leq C_2l(B)\pt{\fint_B \abs{\nabla u}^pw\dx}^{\frac 1p}
		\end{equation}
		for bounded \(u\in C^1(B)\).
		
		\item 
		\emph{Local Sobolev inequality}: There exist constants \(C_2>0\) and \(\chi_1>1\) (same as above) such that for all balls \(B\subset\Omega\) one has
		\begin{equation}\label{sob-ineq_w1}
			\pt{\fint_B\abs{u}^{\chi_1 p}w_1\dx}^{\frac1{\chi_1 p}}\leq C_2l(B)\pt{\fint_B \abs{\nabla u}^pw_1\dx}^{\frac 1p}
		\end{equation}
		for \(u\in C^1_c(B)\).
		
	\end{enumerate}
	
	\begin{remark}
		
		As we mentioned in \cite{Cas2023} the value of \(\chi_1\) comes from a dimensional constant associated to the weight, namely, it can be seen that if \(w\) is a doubling weight then
		\begin{equation}\label{w-ball-estimate1}
			\frac{w(B_R(y))}{w(B_r(x))}\leq C\pt{\frac{R}r}^{D_w},\qquad \text{for all }0<r\leq R<\infty \text{ with }B_r(x)\subseteq B_R(y)\subseteq\Omega.
		\end{equation}
		for \(D_w=\log_2 \gamma_{w}\), and if we denote \(D_1:=\log_2 \gamma_{w_1}\) then we can take \(\chi_1=\frac{D_1}{D_1-p}\) in \eqref{poin-ineq_w1} and \eqref{sob-ineq_w1}.
	\end{remark}

	Regarding the weight \(w_2\), in addition to satisfy \eqref{doubling-w} and \eqref{inv-loc-int2} (in particular \(w_2\) also satisfies \eqref{w-ball-estimate1} for \(D_2:=\log_2 \gamma_{w_2}\)), we require that the following compatibility condition with the weight \(w_1\) is met: there exists \(q>p\) such that 
	\begin{equation}\label{poincare-condition}
		\frac{r}{R}\pt{\frac{w_2(B_r)}{w_2(B_R)}}^{\frac1q}\leq C\pt{\frac{w_1(B_r)}{w_1(B_R)}}^{\frac1p}.
	\end{equation}
	holds for all balls \(B_r\subset B_R\subset \Omega\). From \cite{FraGutWhe1994} (see also \cite[Theorem 7]{Bjorn2001}) we know that if \(1\leq p<q<\infty\), \(w_1\) is \(p\)-admissible, \(w_2\) is doubling and \eqref{poincare-condition} is satisfied, then the pair of weights \((w_1,w_2)\) satisfy the \((q,p)\)-local Poincaré-Sobolev inequality
	\begin{equation}\label{qp-Poincare}
		\pt{\fint_{B_R}\abs{u-u_{B,w_2}}^qw_2\dx}^{\frac1q}\leq CR\pt{\fint_{B_R}\abs{\nabla u}^pw_1\dx}^{\frac 1p},\quad\forall\, u \in C^1(B_R),
	\end{equation}
	and the \((q,p)\)-local Sobolev inequality
	\begin{equation}\label{qp-Sobolev}
		\pt{\fint_{B_R}\abs{u}^qw_2\dx}^{\frac1q}\leq CR\pt{\fint_{B_R}\abs{\nabla u}^pw_1\dx}^{\frac 1p},\quad\forall\, u \in C^1_c(B_R).
	\end{equation}
	
	\begin{remark}
		As it will be useful later we write \(D=\frac{qp}{q-p}\) and \(\chi_2=\frac{D}{D-p}=\frac{q}{p}\). Notice that this \(D\) comes from \eqref{poincare-condition} and in general it has nothing to do with \(D_2=\log_2 \gamma_{w_2}\), the dimensional constant associated to the doubling weight \(w_2\) mentioned before.
	\end{remark}
	
	In order to establish the main results of this work we recall some definitions regarding weighted spaces. For an admissible weight \(w\) we consider the weighted Lebesgue space
	\[
	L^{p,w}(\Omega)=\set{u:\Omega\to\RR \text{ measurable}: \int_{\Omega}\abs{u}^pw\dx<\infty}
	\]
	equipped with the norm
	\[
	\norm{u}_{p,w}^p=\int_{\Omega}\abs{u}^pw\dx.
	\]
	
	The \(p\)-admissibility of \(w_1\) is useful to have a proper definition for weighted Sobolev spaces: for an open set \(\Omega\subseteq\RR^N\) we define the weighted Sobolev space \(H^{1,p,w_1}(\Omega)\)
	\begin{equation*}\label{def-H}
		H^{1,p,w_1}(\Omega)=\text{the completion of }\set{u\in C^1(\Omega): u,\frac{\partial u}{\partial x_i}\in L^{p,w_1}(\Omega) \text{ for all }i}
	\end{equation*}
	equipped with the norm
	\begin{equation*}\label{sobolev-norm}
		\norm{u}_{1,p,w_1}^p=\norm{u}_{p,w_1}^p+\sum_{i=1}^N\norm{\frac{\partial u}{\partial x_i}}_{p,w_1}^p.
	\end{equation*}
	
	As we mentioned before the goal of this work is to generalize what was done in \cite{Cas2023}, that is to obtain qualitative and quantitative properties of weak solutions to \eqref{bb-eq}. To do so we first study the local regularity of weak solutions the following quasi-linear problem
	\begin{equation}\label{model-eq}
		\left\{
		\begin{aligned}
			\di \A(x,u,\nabla u)&=\B(x,u,\nabla u),&&\text{in }\Omega\subseteq\RR^N\\
			u&\in H^{1,p,w_1}_{loc}(\Omega),
		\end{aligned}
		\right.
	\end{equation}
	where \(\A:\Omega\times\RR\times\RR^N\to\RR^N\) and \(\B:\Omega\times\RR\times\RR^N\to \RR\) are functions verifying the Serrin-like conditions
	\begin{gather}
		\label{hyp-A1}\tag{H1}
		\A(x,u,z)\cdot z\geq w_1(x)\pt{a^{-1}\abs{z}^p-d_1\abs{u}^p-g},\\
		\label{hyp-A2}\tag{H2}
		\abs{\A(x,u,z)}\leq w_1(x)\pt{a\abs{z}^{p-1}+b\abs{u}^{p-1}+e},\\
		\label{hyp-B}\tag{H3}
		\abs{\B(x,u,z)}\leq w_2(x)\pt{c\abs{z}^{p-1}+d_2\abs{u}^{p-1}+f},
	\end{gather}
	for a constant \(a>0\) and measurable functions \(b,c,d_1,d_2,e,f,g:\Omega\to\RR^+\cup\set{0}\) satisfying
	\begin{equation}\label{struc-hyp}\tag{H\(_\ve\)}
		\begin{aligned}
			b,e\in L^{\frac{D_1}{p-1},w_1}(B_2),&\quad
			c\pt{\frac{w_2}{w_1}}^{1-\frac1p}\in L^{\frac{D_1}{1-\ve},w_2}(B_2),\\
			d_1,g \in L^{\frac{D_1}{p-\ve},w_1}(B_2),&\quad
			d_2,f \in L^{\frac{D}{p-\ve},w_2}(B_2).
		\end{aligned}
	\end{equation}
	for some \(0\leq \ve<1\).
	
	With the above into consideration, throughout the rest of this article the functions \(w_1,w_2\) will be a non-negative locally integrable weight functions satisfying \eqref{doubling-w}, \eqref{inv-loc-int2}, \(w_1\) will satisfy the local weighted \((1,p)\)-Poincaré inequality \eqref{weighted-poincare} and the pair \((w_1, w_2)\) will verify the compatibility condition \eqref{poincare-condition}. We will also suppose that \(1<p<\min\set{D_1,D}\).
	
	The first result of this work shows that weak solutions to \eqref{model-eq} are locally bounded.
	
	\begin{theorem}\label{thm-local-bdd}
		Suppose that there exists \(0<\ve<1\) such that \eqref{struc-hyp} is satisfied, then there exists a constant \(C>0\) depending on the norms of \(a,b,c,d_1,d_2\) such that for any weak solution to \eqref{model-eq} in \(B_2\) we have
		\begin{equation*}
			\norm{u}_{L^{\infty}(B_1)}\leq
			C\pt{
				[u]_{p,B_2}
				+
				k
			},
		\end{equation*}
		where
		\begin{equation}\label{def-k}
			k=
			\spt{
				\pt{\fint_{B_2}\abs{e}^{\frac{D_1}{p-1}}w_1}^{\frac{p-1}{D_1}}
				+
				\pt{\fint_{B_2}\abs{f}^{\frac{D}{p-\ve}}w_2}^{\frac{p-\ve}{D}}
			}^{\frac{1}{p-1}}
			+
			\spt{
				\pt{\fint_{B_2}\abs{g}^{\frac{D_1}{p-\ve}}w_1}^{\frac{p-\ve}{D_1}}
			}^{\frac1p}
		\end{equation}
		and for \(s>1\) and \(B\subseteq\Omega\) we write
		\begin{equation}\label{n-norm}
			[u]_{s,B}=\pt{\fint_B \abs{\bar u}^s w_1}^{\frac1{s}}+\pt{\fint_B \abs{\bar u}^s w_2}^{\frac1{s}}
		\end{equation}
	\end{theorem}
	
	\begin{remark}
		We have chosen to exhibit the local regularity results only for the case \(B_1\subset B_2\subset \Omega\) as the general case \(B_R\subseteq B_{2R}\subseteq \Omega\) can be easily obtained by a suitable scaling argument (see \cite{Cas2023} where the computations are done in detail).
	\end{remark}
	
	Next we consider the case \(\ve=0\) and we show that weak solutions are in  \(L^{s,w_i}(B_1)\) for every \(s>p\).
	
	\begin{theorem}\label{thm-local-inte}
		Suppose that \eqref{struc-hyp} is satisfied for \(\ve=0\), then there exists a constant \(C>0\) depending on the norms of \(a,b,c,d_1,d_2\) such that for any weak solution to \eqref{model-eq} in \(B_2\) satisfies
		\begin{equation*}
			[u]_{s,B_1}\leq	C_s\pt{[u]_{p,B_2}+k}
		\end{equation*}
		for every \(s>p\) and \(k\) as in \eqref{def-k}.
	\end{theorem}

	Finally, we show that the Harnack inequality holds for non-negative weak solutions to \eqref{model-eq}.
	
	\begin{theorem}[Harnack]\label{harnack-thm}
		Under the same hypotheses of \cref{thm-local-bdd} with the additional assumption that \(u\) is a non-negative weak solution of \(\di\A=\B\) in \(B_{3}\) then
		\begin{equation*}
			\max_{B_{1}}u\leq C\pt{\min_{B_1}u+k} 
		\end{equation*}
		where \(C\) and \(k\) are as in \cref{thm-local-bdd}.
	\end{theorem}
	
	Finally we return to \eqref{bb-eq} and we obtain a general result regarding the behavior at infinity of solutions. To do that we will suppose that in addition to the above conditions, both weights \(w_1,w_2\) verify global Sobolev inequalities, that is, there exists a constant \(C>0\) such that
	\begin{equation}\label{sob-ineq-w_1}
		\pt{\int_\Omega\abs{u}^{q_1}w_1\dx}^{\frac1{q_1}}\leq C\pt{\int_\Omega\abs{\nabla u}w_1\dx}^{\frac1p}
	\end{equation}
	for \(q_1=\chi_1 p\) and
	\begin{equation}\label{sob-ineq-w_2}
		\pt{\int_\Omega\abs{u}^qw_2\dx}^{\frac1q}\leq C\pt{\int_\Omega\abs{\nabla u}w_1\dx}^{\frac1p}
	\end{equation}
	for \(q\) as in \eqref{poincare-condition}, and all \(u\in C^1_c(\Omega)\). Under these assumptions, and if we define \(D^{1,p,w_1}(\Omega)\) as the closure of \(C^\infty_c(\Omega)\) under the (semi) norm \(\norm{\nabla u}_{p,w_1}\) then \(D^{1,p,w_1}(\Omega)\) embeds continuously into both \(L^{q_1,w_1}(\Omega)\) and  \(L^{q,w_2}(\Omega)\) and we are able to prove
	
	\begin{theorem}[Decay]\label{decay-thm}
		Suppose \(u\in D^{1,p,w_1}(\Omega)\) is a weak solution to \eqref{bb-eq}. Then there exists \(R_0>1\), \(C>0\) and \(\lambda>0\) such that
		\begin{equation*}
			\abs{u(x)}\leq \frac{C}{\abs{x}^{\frac{p}{q_1-p}+\lambda}},
		\end{equation*}
		for all \(\abs{x}>R_0\) in \(\Omega\).
	\end{theorem}
	
	\begin{remark}
		It is important to mention that this decay behavior is not optimal, but it can be used as a starting point to obtain better results. This can be done with the aid of a comparison principle a the construction of a suitable barrier function depending on the weights \(w_1,w_2\). We refer the reader to \cite[Section 4]{Cas2023} where power type weights and monomial weights are considered in the case \(w_1=w_2\).
	\end{remark}
	
	The rest of this article is dedicated to the proofs of the above results. In \cref{sec-apriori} we study \eqref{model-eq} and obtain the proofs of \cref{thm-local-bdd,thm-local-inte,harnack-thm} whereas in \cref{sec-infty} we turn to the proof of \cref{decay-thm}.

	\section{Local estimates}\label{sec-apriori}
	
	Throughout the different proofs in this section we will use the dimensional constants of the weights \(D_i:=D_{w_i}\) as well as the local Sobolev exponents \(q_1:=\frac{D_1p}{D_1-p}\) and \(D=\frac{qp}{q-p}\) for \(q\) given by \eqref{poincare-condition}. With these notations we also have
	\[
	\chi_1=\frac{q_1}{p}=\frac{D_1}{D_1-p}\quad\text{and}\quad \chi_2=\frac{q}{p}=\frac{D}{D-p}
	\]
	
	Following \cite{Serrin1964} (and what we did in \cite{Cas2023}) we define \(F:[k,\infty)\to\RR\) as
	\begin{equation*}\label{func-F}
		F(x)=F_{\alpha,k,l}(x)=\begin{dcases}
			x^{\alpha}&\text{if }k\leq x\leq l,\\
			l^{\alpha-1}\pt{\alpha x-(\alpha -1)l}&\text{if }x>l,
		\end{dcases}
	\end{equation*}
	which is in \(C^1([k,\infty))\) with \(\abs{F'(x)}\leq \alpha l^{\alpha-1}\). We consider \(\bar x=\abs{x}+k\) and \(G:\RR\to\RR\) defined as
	\begin{equation*}\label{func-G}
		G(x)=G_{\alpha,k,l}(x)=\mathrm{sign}(x)\pt{F(\bar x)\abs{F'(\bar x)}^{p-1}-\alpha^{p-1}k^\beta}
	\end{equation*}
	where \(\beta=1+p(\alpha-1)\). Observe that \(G\) is a piecewise smooth function which is linear if \(\abs{x}>l-k\) and that both \(F\) and \(G\) satisfy
	\begin{gather*}
		\abs{G}\leq F(\bar x)\abs{F'(\bar x)}^{p-1}\\
		\bar xF'(\bar x)\leq \alpha F(\bar x)\\
		F'(\bar x)\leq \alpha F(\bar x)^{1-\frac{1}{\alpha}}
	\end{gather*}
	and
	\[
	G'(x)=\begin{dcases}
		\frac{\beta}{\alpha}\abs{F'(\bar x)}^p&\text{if }\abs{x}<l-k,\\
		\abs{F'(\bar x)}^p&\text{if }\abs{x}>l-k.
	\end{dcases}
	\]
	
	Finally, observe that if \(\eta\in C^\infty_c(\Omega)\) and if \(u\in H^{1,p,w_1}_{loc}(\Omega)\) then \(\varphi=\eta^pG(u)\) is a valid test function in
	\[
	\int_\Omega \A(x,u,\nabla u)\nabla\varphi+\B(x,u,\nabla u)\varphi=0
	\]
	thanks to the results in \cite[Chapter 1]{HeKiMa2006} regarding weighted Sobolev spaces for \(p\)-admissible weights.
	
	We can now prove the local boundedness of weak solutions.
	
	\begin{proof}[Proof of \cref{thm-local-bdd}]
		By using \eqref{hyp-A1}-\eqref{hyp-B} we can write
		\begin{equation}\label{hyp-mod}
			\begin{aligned}
				\abs{\A(x,u,z)}&\leq w_1\pt{a\abs{z}^{p-1}+\bar b\bar u^{p-1}},\\
				\A(x,u,z)\cdot z&\geq w_1\pt{\abs{z}^{p}-\bar d_1\bar u^{p}},\\
				\abs{\B(x,u,z)}&\leq w_2\pt{c\abs{z}^{p-1}+\bar d_2\bar u^{p-1}},
			\end{aligned}
		\end{equation}
		where
		\begin{gather*}
			\bar b=b+k^{1-p}e,\\
			\bar d_1=d_1+k^{-p}g,\\
			\bar d_2=d_2+k^{1-p}f,
		\end{gather*}
		and \(\bar u=\abs{u}+k\) for \(k\geq 0\) defined as\footnote{If \(e=f=g=0\) we can take any \(k>0\) and at the very end we can pass to the limit \(k\to 0^+\).}  
		\begin{equation*}
			k=
			\spt{
				\pt{\fint_{B_2}\abs{e}^{\frac{D_1}{p-1}}w_1}^{\frac{p-1}{D_1}}
				+
				\pt{\fint_{B_2}\abs{f}^{\frac{D}{p-\ve}}w_2}^{\frac{p-\ve}{D}}
			}^{\frac{1}{p-1}}
			+
			\spt{
				\pt{\fint_{B_2}\abs{g}^{\frac{D_1}{p-\ve}}w_1}^{\frac{p-\ve}{D_1}}
			}^{\frac1p}.
		\end{equation*}
		Observe that \eqref{struc-hyp} implies that
		\begin{equation}\label{k-esti}
			\fint_{B_2}\abs{\bar b}^{\frac{D_1}{p-1}}w_1\leq C,\qquad
			\fint_{B_2}\abs{\bar d_1}^{\frac{D_1}{p-\ve}}w_1\leq C,\qquad
			\fint_{B_2}\abs{\bar d_2}^{\frac{D}{p-\ve}}w_2\leq C,\qquad
		\end{equation}
		for some constant \(C>0\) depending on the respective local norms of \(b,d_1,d_2,e,f,g\).
		
		For a local weak solution \(u\) and arbitrary \(\eta\in C^\infty_c(B_2)\) we use \(\varphi=\eta^pG(u)\) and with the aid of \eqref{hyp-mod} one can obtain the estimate
		\begin{align*}
			\A\cdot\nabla \varphi+\B\varphi&=\eta^p G'(u)\A\cdot\nabla u+p\eta^{p-1}G(u)\A\cdot\nabla \eta+\eta^pG(u)\B\\
			&\geq \eta^p G'(u)w_1\pt{\abs{\nabla u}^{p}-\bar d_1\bar u^{p}}
			-p\eta^{p-1}\abs{\nabla \eta G(u)}w_1\pt{a\abs{\nabla u}^{p-1}+\bar b\bar u^{p-1}}\\
			&\quad-\eta^p\abs{G(u)}w_2\pt{c\abs{\nabla u}^{p-1}+\bar d_2\bar u^{p-1}}
		\end{align*}
		so that if \(v=F(\bar u)\) one reaches
		\begin{multline}\label{bas-esti0}
			\A\cdot\nabla \varphi+\B\varphi\geq
			\abs{\eta \nabla v}^pw_1
			-pa\abs{v\nabla \eta}\abs{\eta \nabla v}^{p-1}w_1
			-p\alpha^{p-1}\bar b\abs{v\nabla \eta}\abs{\eta v}^{p-1}w_1\\
			-\beta\alpha^{p-1}\bar d_1\abs{\eta v}^{p}w_1
			-c\eta v\abs{\eta \nabla v}^{p-1}w_2
			-\alpha^{p-1}\bar d_2\abs{\eta v}^{p}w_2
		\end{multline}
		
		We integrate over \(B_2\) and divide by \(w_1(B_2)\) to obtain
		\begin{multline*}
			\fint_{B_2}\abs{\eta\nabla v}^pw_1
			\leq
			pa \fint_{B_2}\abs{v\nabla \eta}\abs{\eta\nabla v}^{p-1}w_1
			+
			p\alpha^{p-1} \fint_{B_2}\bar b\abs{v\nabla \eta}\abs{v\eta}^{p-1}w_1\\
			+\beta\alpha^{p-1}\fint_{B_2}\bar d_1\abs{v\eta}^pw_1
			+\frac{1}{w_1(B_2)} \int_{B_2} cv\eta\abs{\eta\nabla v}^{p-1}w_2
			+\frac{\alpha^{p-1}}{w_1(B_2)} \int_{B_2}\bar d_2\abs{v\eta}^pw_2,
		\end{multline*}
		
		but since \(w_2(B_2)=Cw_1(B_2)\) for \(C=C(x_0,w_1,w_2)=\frac{w_2(B_2)}{w_1(B_2)}\)  we can write
		\begin{multline}\label{bas-esti}
			\fint_{B_2}\abs{\eta\nabla v}^pw_1
			\leq pa
			\fint_{B_2}\abs{v\nabla \eta}\abs{\eta\nabla v}^{p-1}w_1
			+p\alpha^{p-1}
			\fint_{B_2}\bar b\abs{v\nabla \eta}\abs{v\eta}^{p-1}w_1\\
			+\beta\alpha^{p-1}\fint_{B_2}\bar d_1\abs{v\eta}^pw_1
			+C\fint_{B_2} cv\eta\abs{\eta\nabla v}^{p-1}w_2
			+C\alpha^{p-1}\fint_{B_2}\bar d_2\abs{v\eta}^pw_2,
		\end{multline}
		and each term on the right hand side can be estimated using \eqref{sob-ineq_w1}, \eqref{qp-Sobolev}, and \eqref{k-esti} as follows: 
		\begin{equation}\label{gath1}
			\fint_{B_2}\abs{v\nabla \eta}\abs{\eta\nabla v}^{p-1}w_1\leq \pt{\fint_{B_2}\abs{v\nabla \eta}^pw_1}^{\frac1p}\pt{\fint_{B_2}\abs{\eta\nabla v}^{p}w_1}^{1-\frac1p},
		\end{equation}
		if \(D_1\) the dimensional constant associated to the weight \(w_1\) then
		\begin{equation}\label{gath2}
			\begin{aligned}
				\fint_{B_2}\bar b\abs{v\nabla \eta}\abs{v\eta}^{p-1}w_1
				&\leq
				\pt{\fint_{B_2}\bar b^{\frac{D_{1}}{p-1}}w_1}^{\frac{p-1}{D_{1}}}
				\pt{\fint_{B_2}\abs{v\nabla \eta}^pw_1}^{\frac1p}
				\pt{\fint_{B_2}\abs{v\eta}^{\chi_1p}w_1}^{\frac{p-1}{\chi_1p}}\\
				&\leq
				C
				\pt{\fint_{B_2}\abs{v\nabla \eta}^pw_1}^{\frac1p}
				\pt{\fint_{B_2}\abs{\nabla(v\eta)}^{p}w_1}^{1-\frac1p},
			\end{aligned}
		\end{equation}
		and
		\begin{equation}\label{gath4}
			\begin{aligned}
				\fint_{B_2}\bar d_1\abs{v\eta}^pw_1
				&=\fint_{B_2}\bar d_1\abs{v\eta}^\ve\abs{v\eta}^{p-\ve}w_1\\
				&\leq
				\pt{\fint_{B_2}\bar d_1^{\frac{D_{1}}{p-\ve}}w_1}^{\frac{p-\ve}{D_1}}
				\pt{\fint_{B_2}\abs{v\eta}^pw_1}^{\frac{\ve}{p}}
				\pt{\fint_{B_2}\abs{v\eta}^{\chi_1p}w_1}^{\frac{p-\ve}{\chi_1p}}\\
				&\leq
				C
				\pt{\fint_{B_2}\abs{v\eta}^pw_1}^{\frac{\ve}{p}}
				\pt{\fint_{B_2}\abs{\nabla(v\eta)}^{p}w_1}^{1-\frac{\ve}{p}},
			\end{aligned}
		\end{equation}
		whereas for \(D=\frac{pq}{q-p}\) and \(\bar c=c\pt{\frac{w_2}{w_1}}^{1-\frac1p}\) we have
		\begin{equation}\label{gath3}
			\begin{aligned}
				\fint_{B_2} cv\eta\abs{\eta\nabla v}^{p-1}w_2
				&=\fint_{B_2} \bar c w_2^{\frac{1-\ve}{D}}\abs{v\eta}^{\ve}w_2^{\frac{\ve}{p}}\abs{v\eta}^{1-\ve}w_2^{\frac{1-\ve}{q}}\abs{\eta\nabla v}^{p-1}w_1^{1-\frac1p}\\
				&\leq
				\pt{\fint_{B_2} \abs{\bar c}^{\frac{D}{1-\ve}}w_2}^{{\frac{1-\ve}{D}}}\\
				&\qquad\times
				\pt{\fint_{B_2} \abs{v\eta}^{p}w_2}^{\frac{\ve}{p}}
				\pt{\fint_{B_2} \abs{v\eta}^{q}w_2}^{\frac{1-\ve}{q}}
				\pt{\fint_{B_2} \abs{\eta\nabla v}^{p}w_1}^{1-\frac1p}\\
				&\leq
				C
				\pt{\fint_{B_2} \abs{v\eta}^{p}w_2}^{\frac{\ve}{p}}
				\pt{\fint_{B_2} \abs{\nabla(v\eta)}^{p}w_1}^{\frac{1-\ve}{p}}
				\pt{\fint_{B_2} \abs{\eta\nabla v}^{p}w_1}^{1-\frac1p},
			\end{aligned}
		\end{equation}
		and
		\begin{equation}\label{gath5}
			\begin{aligned}
				\fint_{B_2}\bar d_2\abs{v\eta}^pw_2
				&=\fint_{B_2}\bar d_2\abs{v\eta}^\ve\abs{v\eta}^{p-\ve}w_2\\
				&\leq
				\pt{\fint_{B_2}\bar d_2^{\frac{D}{p-\ve}}w_2}^{\frac{p-\ve}{D}}
				\pt{\fint_{B_2}\abs{v\eta}^pw_2}^{\frac{\ve}{p}}
				\pt{\fint_{B_2}\abs{v\eta}^{q}w_2}^{\frac{p-\ve}{q}}\\
				&\leq
				C
				\pt{\fint_{B_2}\abs{v\eta}^pw_2}^{\frac{\ve}{p}}
				\pt{\fint_{B_2}\abs{\nabla(v\eta)}^{p}w_1}^{1-\frac{\ve}{p}}.
			\end{aligned}
		\end{equation}
		Therefore \eqref{bas-esti}, \eqref{gath1}, \eqref{gath2}, \eqref{gath4}, \eqref{gath3} and \eqref{gath5} give
		\begin{equation}\label{bas-esti2}
			\begin{aligned}
				\fint_{B_2}\abs{\eta\nabla v}^pw_1
				&\leq
				pa
				\pt{\fint_{B_2}\abs{v\nabla \eta}^pw_1}^{\frac1p}
				\pt{\fint_{B_2}\abs{\eta\nabla v}^{p}w_1}^{1-\frac1p}\\
				&+Cp\alpha^{p-1}\spt{
					\pt{\fint_{B_2}\abs{v\nabla\eta}^{p}w_1}+\pt{\int_{B_2}\abs{v\nabla \eta}^pw_1}^{\frac1p}
					\pt{\fint_{B_2}\abs{\eta\nabla v}^{p}w_1}^{1-\frac1p}
				}\\
				&+C\beta\alpha^{p-1}
				\pt{\fint_{B_2}\abs{v\eta}^pw_1}^{\frac{\ve}{p}}
				\spt{
					\pt{\fint_{B_2}\abs{v\nabla \eta}^{p}w_1}^{1-\frac{\ve}{p}}
					+
					\pt{\fint_{B_2}\abs{\eta \nabla v}^{p}w_1}^{1-\frac{\ve}{p}}
				}\\
				&+C\pt{\fint_{B_2} \abs{v\eta}^{p}w_2}^{\frac{\ve}{p}}\\
				&\qquad \times
				\spt{
					\pt{\fint_{B_2} \abs{\eta\nabla v}^{p}w_1}^{1-\frac1p}\pt{\int_{B_2} \abs{v\nabla\eta}^{p}w_1}^{\frac{1-\ve}{p}}
					+
					\pt{\fint_{B_2} \abs{\eta\nabla v}^{p}w_1}^{1-\frac{\ve}p}
				}
				\\
				&+C\alpha^{p-1}
				\pt{\fint_{B_2}\abs{v\eta}^pw_2}^{\frac{\ve}{p}}
				\spt{
					\pt{\fint_{B_2}\abs{v\nabla\eta}^{p}w_1}^{1-\frac{\ve}{p}}
					+
					\pt{\fint_{B_2}\abs{\eta\nabla v}^{p}w_1}^{1-\frac{\ve}{p}}
				}.
			\end{aligned}
		\end{equation}
		If one considers
		\[
		z=\frac{\pt{\fint_{B_2}\abs{\eta\nabla v}^pw_1}^{\frac1p}}{\pt{\fint_{B_2}\abs{v\nabla \eta}^pw_1}^{\frac1p}}
		\]
		and
		\[
		\zeta=\frac{\pt{\fint_{B_2}\abs{\eta v}^pw_1}^{\frac1p}+\pt{\fint_{B_2}\abs{\eta v}^pw_2}^{\frac1p}}{\pt{\fint_{B_2}\abs{v\nabla \eta}^pw_1}^{\frac1p}}
		\]
		then, because \(\alpha\geq 1\), \eqref{bas-esti2} becomes
		\[
		z^p\leq C\pt{z^{p-1}+\alpha^{p-1}(1+z^{p-1})+\zeta^\ve(z^{p-1}+z^{p-\ve})+(1+\beta)\alpha^{p-1}\zeta^\ve(1+z^{p-\ve})}
		\]
		for some constant \(C>0\) depending on \(a,b,c,d,e,f,g, w_1, w_2\) and \(p\). With the aid of \cite[Lemma 2]{Serrin1964} we obtain
		\[
		z\leq C\alpha^{\frac{p}{\ve}}(1+\zeta)
		\]
		which gives
		\begin{equation}\label{fund-esti}
			\pt{\fint_{B_2}\abs{\eta \nabla v}^pw_1}^{\frac1p}
			\leq C\alpha^{\frac{p}{\ve}}
			\pt{
				\pt{\fint_{B_2}\abs{v\nabla \eta}^pw_1}^{\frac1p}
				+\pt{\fint_{B_2}\abs{\eta v}^pw_1}^{\frac1p}
				+\pt{\fint_{B_2}\abs{\eta v}^pw_2}^{\frac1p}
			}.
		\end{equation}
		
		Now, by \eqref{sob-ineq_w1} and \eqref{qp-Sobolev}, that is the local Sobolev inequalities for the pair \((w_1,w_1)\) and the pair \((w_1,w_2)\) respectively we obtain
		\begin{equation}\label{fund-esti2}
			\pt{\fint_{B_2}\abs{\eta v}^{\chi_i p}w_i}^{\frac1{\chi_i p}}
			\leq C\alpha^{\frac{p}{\ve}}
			\pt{
				\pt{\fint_{B_2}\abs{v\nabla \eta}^pw_1}^{\frac1p}
				+\pt{\fint_{B_2}\abs{\eta v}^pw_1}^{\frac1p}
				+\pt{\fint_{B_2}\abs{\eta v}^pw_2}^{\frac1p}
			},
		\end{equation}
		where we recall that \(\chi_1=\frac{D_1}{D_1-p}\) and \(\chi_2=\frac{q}{p}=\frac{D}{D-p}\).
		
		To continue we consider a sequence of cut-off functions as follows: we take \(\eta_n\in C^\infty_c(B_{h_n})\) such that \(\eta_n\equiv 1\) in \(B_{h_{n+1}}\) and \(\abs{\nabla \eta_n}\leq C2^n\) where \(h_n=1+2^{-n}\). If one recalls that both weights are doubling so that \(w_i(B_{h_n})\leq \gamma_{w_i}w_i(B_{h_{n+1}})\) we deduce from \eqref{fund-esti2} that (after passing to the limit \(l\to \infty\))
		\begin{multline}\label{moser-step0}
			\pt{\fint_{B_{h_{n+1}}}\abs{\bar u}^{\alpha\chi_1 p}w_1}^{\frac1{\chi_1 p}}
			+
			\pt{\fint_{B_{h_{n+1}}}\abs{\bar u}^{\alpha\chi_2 p}w_2}^{\frac1{\chi_2 p}}
			\leq
			C2^n\alpha^{\frac{p}{\ve}}
			\left[
			\pt{\fint_{B_{h_n}}\abs{\bar u}^{\alpha p}w_1}^{\frac1p}
			\right.
			\\
			\left.+
			\pt{\fint_{B_{h_n}}\abs{\bar u}^{\alpha p}w_2}^{\frac1p}
			\right],
		\end{multline}
		which is valid for all \(\alpha\geq 1\). Recall the definition of \([u]_{s,B}\) given by \eqref{n-norm}, that is,
		\[
		[u]_{s,B}=\pt{\fint_B \abs{\bar u}^s w_1}^{\frac1{s}}+\pt{\fint_B \abs{\bar u}^s w_2}^{\frac1{s}}
		\]
		and observe that if \(\chi=\min\set{\chi_1,\chi_2}\) then
		\[
		\pt{\fint_{B_{h_{n+1}}}\abs{\bar u}^{\chi^{n+1}p}w_i}^{\frac{1}{\chi^{n+1}p}}
		\leq 
		\pt{\fint_{B_{h_{n+1}}}\abs{\bar u}^{\chi^{n}\chi_ip}w_i}^{\frac{1}{\chi^n\chi_i p}},
		\]
		for \(i=1,2\). Therefore, if we select \(\alpha_n=\chi^n>1\) in \eqref{moser-step0} we are led to
		\[
		[\bar u]_{s_{n+1},B_{h_{n+1}}}\leq C^{\chi^{-n}}2^{n\chi^{-n}}\chi^{\frac{p}{\ve}n\chi^{-n}}[\bar u]_{s_{n},B_{h_n}}
		\]
		where \(s_n=p\chi^n\). And because \(\chi>1\) then \(\sum_{k=0}^\infty k\chi^{-k}\) and \(\sum_{k=0}^\infty\chi^{-k}\) are convergent series so we can iterate the above inequality to obtain
		\[
		[\bar u]_{s_{n},B_{h_n}}\leq C[\bar u]_{p,B_2},
		\]
		for some constant \(C\) independent of \(n\). After passing to the limit \(n\to\infty\) we obtain
		\[
		\norm{u}_{L^{\infty}(B_1)}\leq C\spt{
			\pt{\fint_{B_{2}}\abs{u}^pw_1}^{\frac1p}
			+
			\pt{\fint_{B_{2}}\abs{u}^pw_2}^{\frac1p}
			+
			k
		}
		,
		\]
		and the result follows.
	\end{proof}
	
	\begin{proof}[Proof of \cref{thm-local-inte}]
		Thanks to the interpolation inequality in \(L^{s,w_i}\), it is enough to find a sequence \(s_n\sublim\limits_{n\to\infty}+\infty\) for which one has
		\[
		[\bar u]_{s_n,B_1}\leq C_n[\bar u]_{p,B_2},
		\]
		where \(\bar u=\abs{u}+k\). As in the proof of \cref{thm-local-bdd}, by using the test function \(\varphi=\eta^pG(u)\) we reach to the inequality
		\begin{multline*}
			\fint_{B_2}\abs{\eta\nabla v}^pw_1
			\leq 
			ap\fint_{B_2}\abs{v\nabla \eta}\abs{\eta\nabla v}^{p-1}w_1
			+p\alpha^{p-1}\fint_{B_2}\bar b\abs{v\nabla \eta}\abs{v\eta}^{p-1}w_1
			+\beta\alpha^{p-1}\fint_{B_2}\bar d_1\abs{v\eta}^pw_1\\
			+\fint_{B_2} cv\eta\abs{\eta\nabla v}^{p-1}w_2
			+\alpha^{p-1}\fint_{B_2}\bar d_2\abs{v\eta}^pw_2,
		\end{multline*}
		but because \(\ve=0\) we cannot repeat \eqref{gath1}-\eqref{gath4}. Instead we firstly estimate the term involving \(\bar b\) as follows
		\begin{align*}
			\fint_{B_2}\bar b\abs{v\nabla \eta}\abs{v\eta}^{p-1}w_1&\leq\pt{\fint_{B_2}\bar b^{\frac{D_1}{p-1}}w_1}^{\frac{p-1}{D_1}}\pt{\fint_{B_2}\abs{v\nabla \eta}^pw_1}^{\frac1p}\pt{\fint_{B_2}\abs{v\eta}^{\chi_1 p}w_1}^{\frac{p-1}{\chi_1 p}}\\
			&\leq C\pt{\fint_{B_2}\abs{v\nabla \eta}^pw_1}^{\frac1p}\spt{\pt{\fint_{B_2}\abs{v\nabla \eta}^{p}w_1}^{1-\frac{1}{p}}+\pt{\fint_{B_2}\abs{\eta\nabla v}^{p}w_1}^{1-\frac{1}{p}}}.
		\end{align*}
		For the terms involving \(c\) and \(\bar d\) we consider \(\bar c=c \pt{\frac{w_2}{w_1}}^{1-\frac1p}\) and for each \(M>0\) we define the set \(\mathcal C_M=\set{\bar c\leq M}\) and proceed as follows
		\begin{align*}
			\fint_{B_2} cv\eta\abs{\eta\nabla v}^{p-1}w_2&=
			\frac{1}{w_2(B_2)}
			\left[
			\int_{B_2\cap\mathcal C_M} cv\eta\abs{\eta\nabla v}^{p-1}w_2\right.\\
			&\left.\quad+
			\int_{B_2\cap\mathcal C_M^c} \bar c w_2^{\frac{1}{D}}\abs{v\eta}w_2^{\frac{1}{q}}\abs{\eta\nabla v}^{p-1}w_1^{1-\frac1p}
			\right]\\
			&\leq M
			\pt{\fint_{B_2} \abs{v\eta}^pw_2}^{\frac1p}
			\pt{\fint_{B_2} \abs{\eta\nabla v}^pw_1}^{1-\frac1p}\\
			&\quad+
			\pt{\frac{1}{w_2(B_2)}\int_{B_2\cap\mathcal C_M^c}\abs{\bar c}^{D}w_2}^{\frac1{D}}
			\pt{\fint_{B_2} \abs{v\eta}^{q}w_2}^{\frac1q}
			\pt{\fint_{B_2} \abs{\eta\nabla v}^pw_1}^{1-\frac1p}\\
			&\leq M
			\pt{\fint_{B_2} \abs{v\eta}^pw_2}^{\frac1p}
			\pt{\fint_{B_2} \abs{\eta\nabla v}^pw_1}^{1-\frac1p}\\
			&\quad+C
			\pt{\fint_{B_2}\abs{v\nabla\eta}^pw_1}^{\frac1p}
			\pt{\fint_{B_2} \abs{\eta\nabla v}^pw_1}^{1-\frac1p}\\
			&\quad+C
			\pt{\frac{1}{w_2(B_2)}\int_{B_2\cap \mathcal C_M^c}\abs{\bar c}^Dw_2}^{\frac1{D}}
			\pt{\fint_{B_2}\abs{\eta \nabla v}^pw_1}.
		\end{align*}
		Similarly
		\begin{align*}
			\fint_{B_2}\bar d_1\abs{v\eta}^pw_1&=
			\frac{1}{w_1(B_2)}\spt{
				\int_{B_2\cap\set{\bar d_1\leq M}}\bar d\abs{v\eta}^pw_1
				+
				\int_{B_2\cap\set{\bar d_1>M}}\bar d_1\abs{v\eta}^pw_1
			}
			\\
			&\leq M
			\fint_{B_2}\abs{v\eta}^pw_1
			+
			\pt{\frac{1}{w_1(B_2)}\int_{B_2\cap\set{\bar d_1>M}}\bar d_1^{\frac{D_1}{p}}w_1}^{\frac{p}{D_1}}
			\pt{\fint_{B_2}\abs{v\eta}^{\chi_1 p}w_1}^{\frac1{\chi_1 p}}\\
			&\leq M
			\fint_{B_2}\abs{v\eta}^pw_1
			+C
			\pt{\fint_{B_2}\abs{v \nabla\eta}^{p}w_1}\\
			&\quad+
			\pt{\frac{1}{w_1(B_2)}\int_{B_2\cap\set{\bar d_1>M}}\bar d_1^{\frac{D_1}{p}}w_1}^{\frac{p}{D_1}}
			\pt{\fint_{B_2}\abs{\eta \nabla v}^{p}w_1},
		\end{align*}
		and
		\begin{align*}
			\fint_{B_2}\bar d_2\abs{v\eta}^pw_2
			&=\spt{
				\frac{1}{w_2(B_2)}\int_{B_2\cap\set{\bar d_2\leq M}}\bar d_2\abs{v\eta}^pw_2
				+
				\int_{B_2\cap\set{\bar d_2>M}}\bar d_2\abs{v\eta}^pw_2
			}\\
			&\leq M
			\fint_{B_2}\abs{v\eta}^pw_2
			+
			\pt{\frac{1}{w_2(B_2)}\int_{B_2\cap\set{\bar d_2>M}}\bar d_2^{\frac{D}{p}}w_2}^{\frac{p}{D}}
			\pt{\fint_{B_2}\abs{v\eta}^{q}w_2}^{\frac1{q}}\\
			&\leq M
			\fint_{B_2}\abs{v\eta}^pw_2
			+C
			\pt{\fint_{B_2}\abs{v \nabla\eta}^{p}w_1}\\
			&\quad+
			\pt{\frac{1}{w_2(B_2)}\int_{B_2\cap\set{\bar d_2>M}}\bar d_2^{\frac{D}{p}}w_2}^{\frac{p}{D}}
			\pt{\fint_{B_2}\abs{\eta \nabla v}^{p}w_1}.
		\end{align*}
		
		Because \(\bar c\in L^{D,w_2}\), \(\bar d_1\in L^{\frac{D_1}{p},w_1}\) and \(\bar d_2\in L^{\frac{D}{p},w_2}\) then for any \(\delta>0\) we can find \(M>0\) such that
		\begin{multline*}
			\pt{\frac{1}{w_2(B_2)}\int_{B_2\cap \mathcal C_M}\abs{\bar c}^{D}w_2}^{\frac1{D}}
			+
			\pt{\frac{1}{w_1(B_2)}\int_{B_2\cap\set{\bar d_1>M}}\bar d_1^{\frac{D_1}{p}}w_1}^{\frac{p}{D_1}}\\
			+
			\pt{\frac{1}{w_2(B_2)}\int_{B_2\cap\set{\bar d_2>M}}\bar d_2^{\frac{D}{p}}w_2}^{\frac{p}{D}}
			\leq \delta,
		\end{multline*}
		therefore for any \(\alpha\geq 1\) we can find \(\delta>0\) sufficiently small and a constant \(C_\alpha>0\) such that
		\begin{multline*}
			\fint_{B_2}\abs{\eta\nabla v}^pw_1\leq
			C_\alpha
			\spt{
				\pt{\fint_{B_2}\abs{v\nabla \eta}^pw_1}^{\frac1p}
				+
				\pt{\fint_{B_2}\abs{v\eta}^pw_2}^{\frac1p}
			}
			\pt{\fint_{B_2}\abs{\eta\nabla v}^{p}w_1}^{1-\frac1p}\\
			+
			C_\alpha
			\pt{\fint_{B_2}\abs{v\nabla \eta}^pw_1}
			+C_\alpha
			\pt{\fint_{B_2}\abs{v\eta}^pw_1}
			+C_\alpha
			\pt{\fint_{B_2}\abs{v\eta}^pw_2}.
		\end{multline*}
		
		The above inequality allows us to we use \cite[Lemma 2]{Serrin1964} once again and obtain an inequality analogous to \eqref{fund-esti}, namely
		\begin{equation}\label{fund-esti4}
			\pt{\fint_{B_2}\abs{\eta \nabla v}^pw_1}^{\frac1p}
			\leq
			C_\alpha
			\spt{
				\pt{\fint_{B_2}\abs{v\nabla \eta}^pw_1}^{\frac1p}
				+
				\pt{\fint_{B_2}\abs{\eta v}^pw_1}^{\frac1p}
				+
				\pt{\fint_{B_2}\abs{\eta v}^pw_2}^{\frac1p}
			}
		\end{equation}
		the main difference being that the constant \(C_\alpha\) is no longer explicit. Nonetheless we can continue the argument from the proof of \cref{thm-local-bdd} by choosing appropriate cut-off functions \(\eta\) to reach
		\[
		[\bar u]_{s_{n+1},B_{h_{n+1}}}\leq C_n[\bar u]_{s_n,B_{h_n}},
		\]
		where \(s_n=p\chi^n\), \(h_n=1+2^{-n}\) and \([u]_{s,B}\) is defined in \eqref{n-norm}. Observe that while we do not obtain a uniform estimate for \(C_n\) we can still iterate the above to conclude that
		\[
		[\bar u]_{s_n,B_1}\leq C_n[\bar u]_{p,B_2}
		\]
		and the result is proved.
	\end{proof}
	
	\begin{proof}[Proof of \cref{harnack-thm}]
		\cref{thm-local-bdd} says that \(u\) is bounded on any compact subset of \(B_{3}\) hence for any \(\beta\in\RR\) and any \(\delta>0\) the function \(\varphi=\eta^p\bar u^{\beta}\) is a valid test function provided \(\bar u=u+k+\delta\) and \(\eta\in C^\infty_c(B_3)\). Here \(k\) is defined exactly as in \cref{thm-local-bdd}.
		
		For \(\beta=1-p\) and \(v=\log \bar u\) this is similar to what we did in \cite{Cas2023}, the main difference is the appearance of the weight \(w_2\). We obtain
		\begin{multline}\label{esti-jn-1}
			(p-1)\int_{B_3} \abs{\eta\nabla v}^pw_1\leq pa\int_{B_3} \abs{\nabla \eta}\abs{\eta\nabla v}^{p-1}w_1+p\int_{B_3} \bar b\eta^{p-1}\abs{\nabla\eta}w_1+\int_{B_3} c\eta\abs{\eta\nabla v}^{p-1}w_2\\
			+(p-1)\int_{B_3} \bar d_1\eta^pw_1+\int_{B_3} \bar d_2\eta^pw_2,
		\end{multline}
		for any \(\eta\in C^\infty_c(B_3)\). To continue denote by \(z=\pt{\int_{B_3}\abs{\eta\nabla v}^pw_1}^{\frac1p}\) and with the aid of Hölder's inequality \eqref{esti-jn-1} becomes
		\[
		z^{p}\leq C_1z^{p-1}+C_2,
		\]
		where for \(\bar c=c\pt{\frac{w_2}{w_1}}^{1-\frac1p}\) we have
		\begin{gather}
			C_1=\frac{pa}{p-1}\pt{\int_{B_3}\abs{\nabla\eta}^pw_1}^{\frac1p}+\frac1{p-1}\pt{\int_{B_3}\abs{\bar c \eta}^pw_2}^{\frac1p},\label{const-11}\\
			C_2=\frac{p}{p-1}\int_{B_3}\bar b\eta^{p-1}\abs{\nabla\eta}w_1+\int_{B_3}\bar d_1\eta^pw_1+\frac{1}{p-1}\int_{B_3}\bar d_2\eta^pw_2,\label{const-12}
		\end{gather}
		which thanks to Young's inequality imply
		\[
		z^p\leq C(C_1^p+C_2),
		\]
		for some constant \(C\). To continue we estimate \(C_1\) and \(C_2\) using appropriate \(\eta\). For any \(0<h<2\) such that \(B_h\subset B_2\) (not necessarily concentric) we have that \(B_{\frac{3h}2}\subset B_3\) and we consider \(\eta\in C^\infty_c(B_{\frac{3h}2})\) such that \(\eta\equiv 1\) in \(B_h\), \(0\leq \eta\leq 1\) and \(\abs{\nabla \eta}\leq Ch^{-1}\).
		
		We use such \(\eta\) in \eqref{const-11}-\eqref{const-12} and we get the following estimates using Hölder inequality and the properties of \(\eta\)
		\begin{equation*}
			\int_{B_3} \abs{\nabla \eta}^pw_1\leq \frac{C}{h^p}w_1(B_{\frac{3h}2}),
		\end{equation*}
		\begin{align*}
			\int_{B_3} \bar b\eta^{p-1}\abs{\nabla\eta}w_1&\leq \frac{C}{h}w_1(B_{\frac{3h}{2}})^{1-\frac{p-1}{D_1}}\pt{\int_{B_3} \abs{\bar b}^{\frac{D_1}{p-1}}w_1}^{\frac{p-1}{D_1}},
		\end{align*}
		\begin{align*}
			\int_{B_3} \abs{\bar c \eta }^pw_2
			&\leq Cw_2(B_{\frac{3h}2})^{1-\frac{(1-\ve)p}{D}}\pt{\int_{B_3} \abs{\bar c}^{\frac{D}{1-\ve}}w_2}^{\frac{(1-\ve)p}{D}},
		\end{align*}
		\begin{align*}
			\int_{B_3} \bar d_1\eta^pw_1&\leq Cw_1(B_{\frac{3h}2})^{1-\frac{p-\ve}{D_1}}\pt{\int_{B_3} \abs{\bar d_1}^{\frac{D_1}{p-\ve}}w}^{\frac{p-\ve}{D_1}},
		\end{align*}
		\begin{align*}
			\int_{B_3} \bar d_2\eta^pw_2
			&\leq Cw_2(B_{\frac{3h}2})^{1-\frac{p-\ve}{D}}\pt{\int_{B_3} \abs{\bar d_2}^{\frac{D}{p-\ve}}w}^{\frac{p-\ve}{D}}.
		\end{align*}
		Therefore one obtains
		\begin{align*}
			h^p\fint_{B_h}\abs{\nabla v}^pw_1&\leq \frac{h^p}{w_1(B_h)}\int_{B_3}\abs{\eta\nabla v}^pw_1\\
			&\leq \frac{Ch^p}{w_1(B_h)}\pt{C_1^p+C_2}\\
			&\leq C\left(
			\frac{w_1(B_{\frac{3h}{2}})}{w_1(B_h)}
			+h^{p-1}\frac{w_1(B_{\frac{3h}{2}})^{1-\frac{p-1}{D_1}}}{w_1(B_h)}
			+h^p\frac{w_1(B_{\frac{3h}2})^{1-\frac{p-\ve}{D_1}}}{w_1(B_h)}
			\right.\\
			&\left.
			\qquad
			+h^p\frac{w_2(B_{\frac{3h}2})^{1-\frac{(1-\ve)p}{D}}}{w_1(B_h)}
			+h^p\frac{w_2(B_{\frac{3h}2})^{1-\frac{p-\ve}{D}}}{w_1(B_h)}
			\right),
		\end{align*}
		where \(C\) depends on \(\int_{B_3}\abs{\bar b}^{\frac{D_1}{p-1}}w_1,\ \int_{B_3}\abs{\bar c}^{\frac{D}{1-\ve}}w_2\), \(\int_{B_3}\abs{\bar d_1}^{\frac{D_1}{p-\ve}}w_1\), and \(\int_{B_3}\abs{\bar d_2}^{\frac{D}{p-\ve}}w\). We claim that the right hand side of the above inequality is bounded independently of \(0<h\leq 2\), indeed because \(w_1\) is doubling we have
		\[
		\frac{w_1(B_{\frac{3h}{2}})}{w_1(B_h)}\leq C,
		\]
		and also because \(B_{\frac{3h}{2}}\subset B_3\) we deduce from \eqref{w-ball-estimate1} that \(Ch^{D_1}w_1(B_3)\leq w_1(B_{\frac{3h}{2}})\), hence
		\[
		h^{p-1}\frac{w_1(B_{\frac{3h}{2}})^{1-\frac{p-1}{D_1}}}{w_1(B_h)}\leq \frac{\gamma_{w_1}h^{p-1}}{w_1(B_{\frac{3h}{2}})^{\frac{p-1}{D_1}}}\leq C,
		\]
		also
		\[
		h^p\frac{w_1(B_{\frac{3h}2})^{1-\frac{p-\ve}{D_1}}}{w_1(B_h)}\leq \frac{\gamma_{w_1}h^p}{w_1(B_{\frac{3h}2})^{\frac{p-\ve}{D_1}}}\leq Ch^\ve.
		\]
		From \eqref{poincare-condition} we deduce
		\begin{align*}
			h^p\frac{w_2(B_{\frac{3h}2})^{1-\frac{(1-\ve)p}{D}}}{w_1(B_h)}
			&= h^p\frac{w_2(B_{\frac{3h}2})^{\frac{p}{q}+\ve\pt{1-\frac{p}q}}}{w_1(B_h)}\\
			&= h^p
			\pt{\frac{w_2(B_{\frac{3h}2})^{\frac{1}{q}}}{w_1(B_h)^{\frac{1}{p}}}}^p
			w_2(B_{\frac{3h}2})^{\ve\pt{1-\frac{p}q}}\\
			&\leq \gamma_{w_2}^{\frac{p}{q}} h^p
			\pt{\frac{w_2(B_{h})^{\frac{1}{q}}}{w_1(B_h)^{\frac{1}{p}}}}^p
			w_2(B_3)^{\ve\pt{1-\frac{p}q}}\\
			&\leq \gamma_{w_2}^{\frac{p}{q}} h^p
			\pt{C\pt{\frac3h}\frac{w_2(B_{3})^{\frac{1}{q}}}{w_1(B_3)^{\frac{1}{p}}}}^p
			w_2(B_3)^{\ve\pt{1-\frac{p}q}}\\
			&\leq C
		\end{align*}
		and similarly
		\[
		h^p\frac{w_2(B_{\frac{3h}2})^{1-\frac{p-\ve}{D}}}{w_1(B_h)}=h^p\pt{\frac{w_2(B_{\frac{3h}2})^{\frac{1}{q}}}{w_1(B_h)^{\frac{1}{p}}}}^pw_2(B_{\frac{3h}2})^{\frac{\ve}{p}\pt{1-\frac{p}q}}
		\leq C.
		\]
		Hence for any \(\ve\geq 0\) each term on the right hand side is bounded independently of \(0<h\leq 2\).
		
		Finally, the local Poincaré-Sobolev inequalities \eqref{poin-ineq_w1} and \eqref{qp-Poincare} tell us that
		\begin{align*}
			\fint_{B_h}\abs{v-v_{B_h}}w_i&\leq \pt{\fint_{B_h}\abs{v-v_{B_h}}^{q_i}w_i}^{\frac1{q_i}}\\
			&\leq Ch\pt{\fint_{B_h}\abs{\nabla v}^pw_1}^{\frac1p}\\
			&\leq C,
		\end{align*}
		for any ball \(B_h\subseteq B_2\) and both \(i=1, 2\). We conclude that
		\begin{equation}\label{BMO-esti}
			\fint_{B_h}\abs{v-v_{B_h}}w_i\leq C
		\end{equation}
		where \(C>0\) is a constant not depending on \(h\), in other words, \(v\in \mathrm{BMO}(B_2,w_i\dx)\). If we denote by \(\norm{v}_{BMO(B_2,w_i)}\) as the least possible \(C>0\) in \eqref{BMO-esti} then the John-Nirenberg lemma for doubling measures \cite[Appendix II]{HeKiMa2006} tells us that there exist constants \(p_{0,i},C>0\) such that
		\[
		\fint_{B}e^{p_{0,i}\abs{v-v_{B}}}w_i\leq C
		\]
		for all balls \(B\subseteq B_2\). In particular this gives
		\[
		\pt{\fint_{B_2}e^{p_{0,i}v}w_i} \cdot \pt{\fint_{B_2}e^{-p_{0,i}v}w_i} \leq C^2,
		\]
		and because \(v=\log \bar u\) we have obtained
		\[
		\fint_{B_2}\bar u^{p_{0,i}}w_i\leq C\pt{\fint_{B_2}\bar u^{-p_{0,i}}w_i}^{-1}.
		\]
		Denote by \(p_0=\min\set{p_{0,1},p_{0,2}}\) and observe that
		\[
		\fint_{B_2}\bar u^{p_{0}}w_i\leq C\pt{\fint_{B_2}\bar u^{-p_{0}}w_i}^{-1}.
		\]
		holds for both \(i=1,2\) because \(p_0\leq p_{0,i}\) and Hölder inequality. Therefore if we denote by \(\Psi(p,h)=\pt{\fint_{B_h}\bar u^pw_1}^{\frac1p}+\pt{\fint_{B_h}\bar u^pw_2}^{\frac1p}\) then the above implies
		\begin{equation}\label{esti-37}
			\Psi(p_0,2)\leq C\Psi(-p_0,2).
		\end{equation}
		
		The rest of the proof consists in using \(\varphi=\eta^p\bar u^\beta\) for \(\beta\neq 1-p,\, 0\) as test function and \(v=\bar u^\alpha\) for \(\alpha\) given by \(p\beta=p+\alpha-1\). This gives
		\begin{align*}
			\abs{\alpha}^{p}\pt{\A\cdot \nabla \varphi+\B\varphi} &\geq
			w_1\pt{\beta\abs{\eta \nabla v}^{p}-\beta \abs{\alpha}^{p}\bar d_1\abs{\eta v}^{p}}\\
			&\qquad
			-w_1\pt{ap\abs{\alpha}\abs{\nabla\eta v}\abs{\eta \nabla v}^{p-1}+p\abs{\alpha}^{p}\bar b\abs{\eta v}^{p-1}\abs{\nabla\eta v}}\\
			&\qquad
			-w_2\pt{\abs{\alpha}c\abs{\eta v}\abs{\eta \nabla v}^{p-1}+\abs{\alpha}^{p}\bar d_2\abs{\eta v}^{p-1}}
		\end{align*}
		which after integrating over \(B_3\) becomes
		\begin{align*}
			0 &\geq
			\fint_{B_3} \pt{\beta\abs{\eta \nabla v}^{p}-\beta \abs{\alpha}^{p}\bar d_1\abs{\eta v}^{p}}w_1\\
			&\qquad
			-\fint_{B_3}\pt{ap\abs{\alpha}\abs{\nabla\eta v}\abs{\eta \nabla v}^{p-1}+p\abs{\alpha}^{p}\bar b\abs{\eta v}^{p-1}\abs{\nabla\eta v}}w_1\\
			&\qquad
			-C\fint_{B_3}\pt{c\abs{\alpha}\abs{\eta v}\abs{\eta \nabla v}^{p-1}+\abs{\alpha}^{p}\bar d_2\abs{\eta v}^{p-1}}w_2
		\end{align*}
		where \(C=\frac{w_2(B_3)}{w_1(B_3)}\). Depending on \(\beta\) we have
		
		\begin{itemize}[leftmargin=*]
			\item If \(\beta>0\) then we have
			
			\begin{align*}
				\beta\fint_{B_3} \abs{\eta \nabla v}^{p}w_1 &\leq
				ap\abs{\alpha}\fint_{B_3}\abs{\nabla\eta v}\abs{\eta \nabla v}^{p-1}w_1
				+
				p\abs{\alpha}^{p}\fint_{B_3}\bar b\abs{\eta v}^{p-1}\abs{\nabla\eta v}w_1\\
				&\qquad
				+\beta \abs{\alpha}^{p}\fint_{B_3} \bar d_1\abs{\eta v}^{p}w_1
				+C\abs{\alpha}\fint_{B_3} c\abs{\eta v}\abs{\eta \nabla v}^{p-1}w_2\\
				&\qquad
				+C\abs{\alpha}^{p}\fint_{B_3}\bar d_2\abs{\eta v}^{p-1}w_2
			\end{align*}
			and if we proceed as in the proof of \cref{thm-local-bdd} to estimate each integral on the right hand side we obtain
			\begin{equation*}
				\pt{\fint_{B_3} \abs{\eta \nabla v}^{\chi_i p}w_i}^{\frac1{\chi_i p}}
				\leq C\alpha^{\frac{p}{\ve}}(1+\beta^{-1})^{\frac1\ve}
				\spt{
					\pt{\fint_{B_3} \abs{\eta v}^pw_1}^{\frac1p}
					+\pt{\fint_{B_3} \abs{\eta v}^pw_2}^{\frac1p}
					+\pt{\fint_{B_3} \abs{\nabla\eta v}^pw_1}^{\frac1p}
				}.
			\end{equation*}
			If \(\eta\in C^\infty_c(B_h)\) is such that \(\eta\equiv 1\) in \(B_{h'}\) for \(1\leq h'<h\leq 2\) with \(\abs{\nabla \eta}\leq C(h-h')^{-1}\) then
			\begin{multline*}
				\pt{\fint_{B_{h'}} \abs{v}^{\chi_i p}w_i}^{\frac1{\chi_i p}}
				\leq C\pt{\frac{w_i(B_3)}{w_i(B_{h'})}}^{\frac1{\chi_i p}}\frac{\alpha^{\frac{p}{\ve}}(1+\beta^{-1})^{\frac1\ve}}{h-h'}\\
				\times \spt{\pt{\frac{w_1(B_{h})}{w_1(B_3)}}^{\frac1{p}}\fint_{B_h} \abs{v}^pw_1+\pt{\frac{w_2(B_{h})}{w_2(B_3)}}^{\frac1{p}}\fint_{B_h} \abs{v}^pw_2}^{\frac1p},
			\end{multline*}
			but since \(1\leq h'<h\leq 2\) we have
			\[
			\frac{w_i(B_3)}{w_i(B_{h'})}\leq \frac{w_i(B_{4h'})}{w_i(B_{h'})}\leq \gamma_{w_i}^2\quad \text{and}\quad \frac{w_i(B_h)}{w_i(B_{3})}\leq 1
			\]
			hence for \(\chi=\min\set{\chi_1,\chi_2}\) we have
			\begin{equation}\label{esti-38}
				\Psi(\chi p,h')
				\leq C\frac{\alpha^{\frac{p}{\ve}}(1+\beta^{-1})^{\frac1\ve}}{h-h'}\Psi(p,h).
			\end{equation}
			
			\item Similarly, for \(1-p<\beta<0\) one has
			\begin{equation}\label{esti-39}
				\Psi(\chi p,h')
				\leq C\frac{(1-\beta^{-1})^{\frac1\ve}}{h-h'}\Psi(p,h).
			\end{equation}
			\item If \(\beta<1-p\) then one obtains
			\begin{equation}\label{esti-40}
				\Psi(\chi p',h')
				\leq C\frac{(1+\abs{\alpha})^{\frac{p}\ve}}{h-h'}\Psi(p,h).
			\end{equation}
		\end{itemize}
		
		If we observe that \(\Psi(s,r)\sublim\limits_{s\to\infty} 2\max\limits_{B_r}\bar u\) and \(\Psi(s,r)\sublim\limits_{s\to-\infty} 2\min\limits_{B_r}\bar u\) then we can repeat the iterative argument from the proof of \cite[Theorem 5]{Serrin1964} to deduce that \eqref{esti-38} and \eqref{esti-39} imply
		\[
		\max_{B_1}\bar u\leq C\Psi(p_0',2)
		\]
		for some \(p_0'\leq p_0\) chosen appropriately, whereas \eqref{esti-40} will give
		\[
		\min_{B_1}\bar u\geq C^{-1}\Psi(-p_0,2).
		\]
		Finally we can use \eqref{esti-37} to obtain a constant \(C>0\) depending on the structural parameters such that
		\[
		\max_{B_1}\bar u\leq C\min_{B_1}\bar u
		\]
		and because \(\bar u=u+k+\delta\) we conclude by letting \(\delta\to 0^+\).
	\end{proof}

	\section{Behavior at infinity}\label{sec-infty}
	
	In this section we obtain a decay estimate for weak solutions to the equation
	\begin{equation}\label{extremal-eq}
		\left\{
		\begin{aligned}
			-\di(w_1\abs{\nabla u}^{p-2}\nabla u)&=w_2\abs{u}^{q-2}u&&\text{in }\Omega\\
			u&\in D^{1,p,w_1}(\Omega)&&
		\end{aligned}
		\right.
	\end{equation}
	where the set \(\Omega\subseteq\RR^N\) (bounded or not) is such that there exists a constant  \(C>0\) for which the global weighted Sobolev inequalities \eqref{sob-ineq-w_1} and \eqref{sob-ineq-w_2} hold. With the aid of the results regarding the equation \(\di\A =\B\) we are able to prove that that weak solutions to \eqref{extremal-eq} are locally bounded.
	
	\begin{lemma}\label{ext-bd-local}
		Let \(u\in D^{1,p,w}(\Omega)\) be a weak solution of
		\[
		-\di(w_1\abs{\nabla u}^{p-2}\nabla u)=w_2\abs{u}^{q-2}u\qquad\text{in }\Omega.
		\]
		Then for every \(R>0\) such that \(B_{4R}(x_0)\subseteq\Omega\) then there exists \(C_R>0\) such that
		\[
		\norm{u}_{L^{\infty}(B_R(x_0))} \leq C_R[u]_{p, B_{4R}(x_0)}.
		\]
	\end{lemma}

	\begin{proof}
		Observe that equation \eqref{extremal-eq} can be written in the from \(\di\A=\B\) for \(a=1\), \(b=c=d_1=e=f=g=0\) and \(d_2=-\abs{u}^{q-p}\). We first use \cref{thm-local-inte} because from that result we know that if \(d_2 \in L^{\frac{D}{p},w_2}\) then for every \(s\geq 1\) and \(R>0\) the weak solution \(u\) satisfies 
		\begin{align*}
			\pt{\fint_{B_{2R}(x_0)}\abs{u}^{s}w_1}^{\frac1s}
			+
			\pt{\fint_{B_{2R}(x_0)}\abs{u}^{s}w_2}^{\frac1s}
			&\leq C_{R,s}
			\spt{
				\pt{\fint_{B_{4R}(x_0)}\abs{u}^pw_1}^{\frac1p}
				+
				\pt{\fint_{B_{4R}(x_0)}\abs{u}^pw_2}^{\frac1p}
			},
		\end{align*}
		and \(C_{R,s}\) depends on \(s\) and on \(\pt{\fint_{B_{4R}(x_0)}\abs{d_2}^{\frac{D}{p}}w_2}^{\frac{p}D}\). But because \(u\in D^{1,p,w_1}(\Omega)\) and the weights \(w_1,w_2\) verify \eqref{poincare-condition} then the local Sobolev inequality \eqref{qp-Sobolev} holds and we have that \(u\in L^{q,w_2}(\Omega)\), hence \(d\in L^{\frac{D}{p},w_2}(B_{4R}(x_0))\Leftrightarrow q=\frac{Dp}{D-p}\). In particular, this shows that \(u\in L^{s,w_2}(B_{2R}(x_0))\) for every \(s\) and as a consequence \(d_2=-\abs{u}^{q-p}\in L^{\frac{D}{p-\ve},w_2}(B_{2R}(x_0))\) for every \(0<\ve<p\). Therefore we can now use \cref{thm-local-bdd} to conclude that 
		\[
		\norm{u}_{L^{\infty}(B_R(x_0))}\leq C_{R}[u]_{p, B_{4R}(x_0)},
		\]
		where \(C_R\) depends on \(R>0\) and the norm of \(u\) in \(D^{1,p,w_1}(\Omega)\).
	\end{proof}
	
	Now we would like to estimate the decay of the \(L^{q_1,w_1}\) norm of weak solutions as one leaves the set \(\Omega\).
	
	\begin{lemma}\label{lem-tau}
		Suppose \(u\in D^{1,p,w_1}(\Omega)\) is a weak solution of \eqref{extremal-eq}, then there exists \(R_0>0\) and \(\tau>0\) such that if \(R\geq R_0\) then
		\[
		\norm{u}_{L^{q_1,w_1}(\Omega\setminus B_{R})}\leq \pt{\frac{R_0}{R}}^{\tau}\norm{u}_{L^{q_1,w_1}(\Omega\setminus B_{R_0})}.
		\]
		Here \(B_R\) denotes an arbitrary ball of radius \(R\).
	\end{lemma}

	\begin{proof}
		Because \(u\in D^{1,p,w}(\Omega )\) then for \(\eta\in W^{1,\infty}(\RR^N)\) the function \(\varphi=\eta^p u\) is a valid test function in
		\[
		\int_\Omega \abs{\nabla u}^{p-2}\nabla u\nabla\varphi w_1 = \int_\Omega \abs{u}^{q-2}u\varphi w_2.
		\]
		On the one hand, using Young's inequality we can find \(C_p>0\) such that
		\begin{align*}
			\int_\Omega \abs{\nabla u}^{p-2}\nabla u\nabla\varphi w_1
			&=
			\int_\Omega \abs{\eta\nabla u}^{p} w_1
			+
			p\int_\Omega \eta^{p-1}\abs{\nabla u}^{p-2}\nabla u\cdot u\nabla\eta w_1 \\
			&\geq \frac12
			\int_\Omega \abs{\eta\nabla u}^{p} w_1
			-C_p
			\int_\Omega \abs{u\nabla\eta}^p w_1.
		\end{align*}
		
		On the other hand, since \(q>p\) we can write
		\begin{align*}
			\int_\Omega \abs{u}^{q-2}u\varphi w_2
			&=\int_\Omega u^{q}\eta^p w_2\\
			&=\int_\Omega \abs{u}^{q-p}\abs{\eta u}^p w_2\\
			&\leq
			\pt{\int_{\supp \eta} \abs{u}^{q} w_2}^{1-\frac{p}{q}}
			\pt{\int_\Omega \abs{\eta u}^{q} w_2}^{\frac{p}{q}}.
		\end{align*}
		
		Hence 
		\begin{align*}
			\int_\Omega \abs{\nabla(\eta u)}^{p} w_1
			&=\int_\Omega \abs{\eta\nabla u+u\nabla \eta}^{p} w_1\\
			&\leq 2^{p-1}
			\int_\Omega \abs{\eta\nabla u}^{p} w_1
			+2^{p-1}
			\int_\Omega \abs{u\nabla \eta}^{p} w_1\\
			&\leq 2^{p-1}\pt{
				2\int_\Omega \abs{\nabla u}^{p-2}\nabla u\nabla\varphi w_1
				+C_p
				\int_\Omega \abs{u\nabla\eta}^p w_1
			}
			+2^{p-1}
			\int_\Omega \abs{u\nabla \eta}^{p} w_1\\
			&\leq C_p
			\int_\Omega \abs{u\nabla\eta}^p w_1
			+2^p
			\pt{\int_{\supp \eta} \abs{u}^{q} w_2 }^{1-\frac{p}{q}}
			\pt{\int_\Omega \abs{\eta u}^{q} w_2}^{\frac{p}{q}},\\
		\end{align*}
		and the global Sobolev inequality \eqref{sob-ineq-w_2} tells us that there exists a constant \(C_{p,w_1,w_2}>0\) such that
		\begin{equation}\label{p*bd1}
			\int_\Omega \abs{\nabla(\eta u)}^{p} w_1\leq C_p
			\int_\Omega \abs{u\nabla\eta}^p w_1
			+C_{p,w_1,w_2}
			\pt{\int_{\supp \eta} \abs{u}^{q} w_2 }^{1-\frac{p}{q}}
			\pt{\int_\Omega \abs{\nabla(\eta u)}^{p} w_1}.
		\end{equation}
		
		We now choose \(\eta\). First of all, because \(\norm{u}_{q,w_2}\) is finite for any given \(\ve>0\) we can find \(R_0=R_0(\ve)>0\) such that if \(R\geq R_0\) then 
		\[
		\int_{\Omega\setminus B_{R}}\abs{u}^{q} w_2 \leq \ve.
		\]
		With this in mind we choose \(R_0>0\) such that 
		\[
		C_{p,w_1,w_2}\pt{\int_{\Omega\setminus B_{R_0}}\abs{u}^{q} w_2 }^{1-\frac{p}q}\leq \frac{1}{2},
		\]
		and we suppose that \(R\geq R_0\) from now on. We consider \(\eta\in W^{1,\infty}(\RR^N)\), such that \(0\leq \eta\leq 1\), \(\eta(x)=0\) for \(x\in B_{R}\), \(\eta(x)=1\) for \(x\notin B_{2R}\), and \(\abs{\nabla\eta}\leq CR^{-1}\). If we use such \(\eta\) in \eqref{p*bd1} we obtain a constant \(C>0\) independent of \(R\) such that
		\[
		\int_\Omega \abs{\nabla(\eta u)}^{p} w_1\leq C_p
		\int_\Omega \abs{u\nabla\eta}^p w_1
		\]
		which after using \eqref{sob-ineq-w_1} gives
		\begin{equation}\label{p*bd2}
			\pt{\int_\Omega \abs{\eta u}^{q_1} w_1}^{\frac{1}{q_1}}\leq C\pt{\int_\Omega \abs{u\nabla\eta}^p w_1}^{\frac1{p}}.
		\end{equation}
		
		By the choice of \(\eta\) we also have
		\begin{align}
			\int_\Omega \abs{u\nabla\eta}^p w_1
			&\leq CR^{-p}
			\int_{\Omega\cap B_{2R}\setminus B_{R}} \abs{u}^p w_1 \notag\\
			&\leq CR^{-p}
			\pt{w_1(\Omega\cap B_{2R})}^{1-\frac{1}{\chi_1}}
			\pt{\int_{\Omega\cap B_{2R}\setminus B_{R}} \abs{u}^{q_1} w_1}^{\frac{1}{\chi_1}}\notag\\
			&\leq CR^{-p}
			\pt{
				w_1(\Omega\cap B_{R_0})
				\pt{\frac{2R}{R_0}}^{D_1}}^{1-\frac{1}{\chi_1}
			}
			\pt{\int_{\Omega\cap B_{2R}\setminus B_{R}} \abs{u}^{q_1} w_1}^{\frac{1}{\chi_1}}\notag\\
			&=C
			\pt{\frac{w_1(\Omega\cap B_{R_0})}{R_0^{D_1}}}^{1-\frac{1}{\chi_1}}R^{D_1(1-\frac{1}{\chi_1})-p}
			\pt{\int_{\Omega\cap B_{2R}\setminus B_{R}} \abs{u}^{q_1} w_1}^{\frac{1}{\chi_1}}\notag\\
			&\leq CR^{D_1(1-\frac{1}{\chi_1})-p}
			\pt{\int_{\Omega\cap B_{2R}\setminus B_{R}} \abs{u}^{q_1} w_1}^{\frac{1}{\chi_1}} \notag\\
			&= C
			\pt{\int_{\Omega\cap B_{2R}\setminus B_{R}} \abs{u}^{q_1} w_1}^{\frac{1}{\chi_1}} \label{decay1}
		\end{align}
		where we have used \eqref{w-ball-estimate1} and the fact that \(\frac1{q_1}=\frac1{D_1}-\frac1p\). From \eqref{p*bd2} and \eqref{decay1} we obtain
		\[
		\int_\Omega \abs{\eta u}^{q_1} w_1 \leq C\int_{\Omega\cap B_{2R}\setminus B_{R}} \abs{u}^{q_1} w_1,
		\]
		for some constant \(C>0\) depending on \(p,q_1, R_0\) but independent of \(R\). To continue, observe that since \(\eta\equiv 1\) on \(B_{2R}^c\) we can write
		\begin{align*}
			\int_{\Omega\setminus B_{2R}}\abs{u}^{q_1} w_1
			&\leq \int_{\Omega}\abs{\eta u}^{q_1} w_1\\
			&\leq C
			\int_{\Omega\cap B_{2R}\setminus B_{R}} \abs{u}^{q_1}w_1\\
			&=C
			\int_{\Omega\setminus B_{R}} \abs{u}^{q_1}w_1
			-C
			\int_{\Omega\setminus B_{2R}} \abs{u}^{q_1} w_1,
		\end{align*}
		thus, if \(\theta=\frac{C}{C+1}\in(0,1)\) then we obtain
		\begin{equation*}\label{bdf}
			\int_{\Omega\setminus B_{2R}}\abs{u}^{q_1} w_1 \leq \theta\int_{\Omega\setminus B_{R}} \abs{u}^{q_1} w_1.
		\end{equation*}
		
		Then just as in \cite{Cas2023} one can find \(\tau>0\) such that
		\[
		\int_{\Omega\setminus B_{R}}\abs{u}^{q_1} w_1\leq \pt{\frac{R_0}{R}}^{\tau}\int_{\Omega\setminus B_{R_0}}\abs{u}^{q_1} w_1
		\]
		for \(\tau=-q_1\log_2\theta>0\).
	\end{proof}
	
	\begin{lemma}\label{lema-tau2}
		Suppose that \(u\in D^{1,p,w_1}(\Omega)\) is a weak solution of
		\begin{equation}\label{eq-tau2}
			-\di(w_1\abs{\nabla u}^{p-2}\nabla u)=w_2\abs{u}^{q-2}u\qquad\text{in }\Omega.
		\end{equation}
		Then for each \(s>\max\set{q_1,q}\) there exists \(R_0>0\) (depending on \(s\)) such that if \(R\geq R_0\) then there exists \(C=C(p,q_1,q,w_1,w_2;s)>0\) for which 
		\[
		\norm{u}_{L^{s,w_i}(\Omega\setminus B_{2R})}\leq \frac{C}{R^{\frac{p}{q_1-p}-o_s(1)}}\norm{u}_{L^{q_1,w_1}(\Omega\setminus B_{R})},
		\]
		for both \(i=1,2\), where \(o_s(1)\) is a quantity that goes to \(0\) as \(s\to \infty\).
	\end{lemma}
	
	\begin{proof}
		Firstly notice that thanks to the \(L^{s,w}\) interpolation inequality it is enough to exhibit a sequence \(s_n\sublim\limits_{n\to\infty}+\infty\) for which one has
		\[
		\norm{u}_{L^{s_n,w_i}(\Omega\setminus B_{2R})}\leq \frac{C}{R^{\frac{p}{q_1-p}-o_n(1)}}\norm{u}_{L^{q_1,w_1}(\Omega\setminus B_{R})}.
		\]
		Observe that in the context of \eqref{model-eq} we can view \eqref{eq-tau2} as \(\di\A=\B\) where \(a=1\), \(b=c=d_1=e=f=g=0\) and \(d_2=\bar d_2=-\abs{u}^{q-p}\). The assumption \(u\in D^{1,p,w_1}(\Omega)\) tells us that \(\varphi=\eta^pG(u)\) is valid test function and we can follow the notation of the proof \cref{thm-local-bdd}, in fact, since \(e=f=g=0\) we can further suppose that \(k>0\) is arbitrary in the definition of both \(F\) and \(G\). Starting with \eqref{bas-esti0} we now integrate over \(\Omega\) to obtain
		\[
		\int_{\Omega}\abs{\eta\nabla v}^pw_1\leq p\int_{\Omega}\abs{v\nabla \eta}\abs{\eta\nabla v}^{p-1}w_1+(\alpha-1)\alpha^{p-1}\int_{\Omega}d_2\abs{v\eta}^pw_2,
		\]
		where \(v=F(\bar u)\). From the above we obtain
		\[
		\int_{\Omega}\abs{\nabla(\eta v)}^pw_1\leq C_\alpha\pt{\int_{\Omega}\abs{v\nabla \eta}^pw_1+\int_{\Omega}\abs{u}^{q-p}\abs{v\eta}^pw_2},
		\]
		and with the help of \eqref{sob-ineq-w_2} we can write 
		\begin{align*}
			\int_{\Omega}\abs{u}^{q-p}\abs{v\eta}^pw_2
			&\leq
			\pt{\int_{\supp \eta} \abs{u}^{q} w_2}^{1-\frac{p}{q}}
			\pt{\int_\Omega \abs{v\eta}^{q} w_2}^{\frac{p}{q}}\\
			&\leq
			C_{p,w_1,w_2}\pt{\int_{\supp \eta} \abs{u}^{q} w_2}^{1-\frac{p}{q}}
			\pt{\int_\Omega \abs{\nabla(v\eta)}^{p} w_1},
		\end{align*}
		therefore we have
		\[
		\int_{\Omega}\abs{\nabla(\eta v)}^pw_1
		\leq C_\alpha
		\int_{\Omega}\abs{v\nabla \eta }^pw_1
		+C_{p,\alpha,w_1,w_2}
		\pt{\int_{\supp \eta} \abs{u}^{q} w_2}^{1-\frac{p}{q}}
		\pt{\int_\Omega \abs{\nabla(v\eta)}^{p} w_1}.
		\]
		
		We now select \(\eta\). Because \(u\in D^{1,p,w_1}(\Omega)\) and that \eqref{sob-ineq-w_2} holds then we know that \(u\in L^{q,w_2}(\Omega)\), therefore for any given \(\nu>0\) we can find \(R_0=R_0(\nu)>0\) such that
		\[
		\int_{\Omega\setminus B_R}\abs{u}^{q} w_2 \leq \nu,\qquad\forall\, R\geq R_0.
		\]
		With this in mind we choose \(R_0=R_0(\alpha)>0\) such that 
		\[
		C_{p,\alpha,w_1,w_2}
		\pt{\int_{\Omega\setminus B_R} \abs{u}^{q} w_2}^{1-\frac{p}{q}}\leq \frac{1}{2},
		\]
		and we suppose that \(R\geq R_0\) to obtain that if \(\supp\eta\subset B_{R}^c\) then
		\[
		\int_{\Omega}\abs{\nabla(\eta v)}^pw_1
		\leq C_\alpha
		\int_{\Omega}\abs{v\nabla \eta }^pw_1,
		\]
		and using \eqref{sob-ineq-w_1}, \eqref{sob-ineq-w_2} and passing to the limits \(l\to+\infty\), \(k\to 0^+\) give
		\begin{gather}
			\pt{\int_{\Omega}\abs{\eta u^{\alpha}}^{q_1}w_1}^{\frac1{q_1}}
			\leq C_\alpha
			\pt{\int_{\Omega}\abs{u^\alpha\nabla \eta }^pw_1}^{\frac1p},\label{l2-esti1}\\
			\pt{\int_{\Omega}\abs{\eta u^{\alpha}}^{q}w_2}^{\frac1{q}}
			\leq C_\alpha
			\pt{\int_{\Omega}\abs{u^\alpha\nabla \eta }^pw_1}^{\frac1p}\label{l2-esti2}.
		\end{gather}
		
		We now select \(\eta\): for \(n\geq 0\) we consider \(R_n=R(2-2^{-n})\) and a smooth function \(\eta\) such that \(0\leq \eta\leq 1\), \(\eta(x)=0\) for \(\abs{x}\leq R_{n}\), \(\eta(x)=1\) for \(\abs{x}\geq R_{n+1}\) and satisfies \(\abs{\nabla\eta}\leq \frac{C2^n}{R}\), 
		\begin{gather*}
			\supp\eta\subseteq \Omega\setminus B_{R_n}\\
			\supp\nabla\eta \subseteq \Omega\cap B_{R_n}\setminus B_{R_{n+1}}.
		\end{gather*}
		Therefore if for \(n\geq 1\) we take \(\alpha_n=\pt{\frac{q_1}{p}}^{n}\) in \eqref{l2-esti1} then we obtain
		\[
		\pt{\int_{\Omega\setminus B_{R_{n+1}}}\abs{u}^{\frac{q_1^{n+1}}{p^{n}}} w_1}^{\frac{p^{n}}{q_1^{n+1}}}
		\leq \pt{\frac{C_n}{R}}^{\frac{p^{n}}{q_1^{n}}}
		\pt{\int_{\Omega\setminus B_{R_n}}\abs{u}^{\frac{q_1^{n}}{p^{n-1}}} w_1}^{\frac{p^{n-1}}{q_1^{n}}},
		\]
		or equivalently, if \(s_n=\frac{q_1^{n}}{p^{n-1}}\) and \(\mathcal U_n=\norm{u}_{L^{s_n,w_1}(\Omega\setminus B_{R_{n}})}\), 
		\[
		\mathcal U_{n+1}\leq \frac{\tilde C_n}{R^{\frac{p^{n}}{q_1^n}}}\mathcal U_n,
		\]
		for \(\tilde C_n=C_n^{\pt{\frac{p}{q_1}}^{n}}\), which after iterating gives
		\[
		\mathcal U_n
		\leq \pt{\frac{\prod_{i=1}^{n-1} \tilde C_i}{R^{\sum_{i=1}^{n-1}\pt{\frac{p}{q_1}}^i}}}\mathcal U_1,
		\]
		and since
		\[
		\sum_{i=1}^{n-1}\pt{\frac{p}{q_1}}^i
		=\frac{p}{q_1-p}-\frac{q_1}{q_1-p}\pt{\frac{p}{q_1}}^{n}=\frac{p}{q_1-p}-o_n(1),
		\]
		because \(q_1>p\) we obtain that for any \(s>q_1\)
		\[
		\norm{u}_{L^{s,w_1}(\Omega\setminus B_{2R})}\leq \frac{C_s}{R^{\frac{p}{q_1-p}-o_s(1)}}\norm{u}_{L^{q_1,w_1}(\Omega\setminus B_{R})},
		\]
		because \(\mathcal U_1\leq \norm{u}_{L^{q_1,w_1}(\Omega\setminus B_R)}\), \(\mathcal U_n\geq \norm{u}_{L^{s_n,w_1}(B_{2R})}\).
		
		With the same choice of \(\eta\) and \(\alpha\) in \eqref{l2-esti2} we have
		\begin{align*}
			\pt{\int_{\Omega\setminus B_{R_{n+1}}}\abs{u}^{\frac{q_1^{n}q}{p^{n}}} w_2}^{\frac{p^{n}}{q_1^{n}q}}
			&\leq \pt{\frac{C_n}{R}}^{\frac{p^{n}}{q_1^{n}}}
			\pt{\int_{\Omega\setminus B_{R_n}}\abs{u}^{\frac{q_1^{n}}{p^{n-1}}} w_1}^{\frac{p^{n-1}}{q_1^{n}}}\\
			&=\pt{\frac{C_n}{R}}^{\frac{p^{n}}{q_1^{n}}}\mathcal U_n\\
			&\leq \pt{\frac{\prod_{i=1}^{n} \tilde C_i}{R^{\sum_{i=1}^{n}\pt{\frac{p}{q_1}}^i}}}\mathcal U_1,
		\end{align*}
		and just as before we deduce that
		\[
		\norm{u}_{L^{s,w_2}(\Omega\setminus B_{2R})}
		\leq
		\frac{C_s}{R^{\frac{p}{q_1-p}-o_s(1)}}\norm{u}_{L^{q_1,w_1}(\Omega\setminus B_{R})}
		\]
		for \(s>q\).
	\end{proof}
	
	Now we are in position to prove \cref{decay-thm}:
	
	\begin{proof}[Proof of \cref{decay-thm}]
		Consider the value of \(R_0>0\) given in \cref{lem-tau}, and suppose that \(x\in \Omega\setminus B_{2R_0}\). Fix \(0<r<\frac{R_0}{4}\) so that \(B_r(x)\subseteq\Omega\) and use \cref{ext-bd-local} to obtain
		\[
		\abs{u(x)}\leq \norm{u}_{L^{\infty}(B_r(x))}
		\leq C_r[u]_{p,B_{2r}}
		\leq C_r\spt{
			\pt{\int_{B_{2r}}\abs{u}^sw_1}^{\frac1s}
			+
			\pt{\int_{B_{2r}}\abs{u}^sw_2}^{\frac1s}
		},
		\]
		for any \(s>p\). If we consider \(R=\frac{\abs{x}}4\), then by geometric considerations we deduce that \(B_{2r}(x)\subseteq \Omega\setminus B_{2R}\) hence
		\[
		\pt{\int_{B_{2r}}\abs{u}^sw_i}^{\frac1s}\leq \pt{\int_{\Omega\setminus B_{2R}}\abs{u}^sw_i}^{\frac1s}.
		\]
		
		Now we fix \(s\) large enough so that \(o_s(1)\leq \frac{\tau}{2}\) in \cref{lema-tau2}, where \(\tau>0\) is taken from \cref{lem-tau}, by doing that we obtain
		\begin{align*}
			\norm{u}_{L^{s,w_2}(\Omega\setminus B_{2R})}+\norm{u}_{L^{s,w_1}(\Omega\setminus B_{2R})}
			&\leq
			\frac{C}{R^{\frac{p}{q_1-p}-o_s(1)}}\norm{u}_{L^{q_1,w_1}(\Omega\setminus B_{R})}\\
			&\leq
			\frac{C}{R^{\frac{p}{q_1-p}-\frac{\tau}2}}\norm{u}_{L^{q_1,w_1}(\Omega\setminus B_{R})}\\
			&\leq
			\frac{C}{R^{\frac{p}{q_1-p}-\frac{\tau}{2}}}\pt{\frac{R_0}{R}}^{\tau}\norm{u}_{L^{q_1,w_1}(\Omega\setminus B_{R_0})},
		\end{align*}
		therefore, by putting all together we obtain
		\[
		\abs{u(x)}\leq \frac{CR_0^{\tau}}{R^{\frac{p}{q_1-p}+\frac{\tau}{2}}}\norm{u}_{L^{q_1,w}(\Omega\setminus B_{R_0})}=\frac{C}{\abs{x}^{\frac{p}{q_1-p}+\lambda}},
		\]
		for some constant \(C>0\) independent of \(\abs{x}\geq 2R_0\), and the result is proved for \(\tilde R=2R_0\).
	\end{proof}


\end{document}